\documentclass{amsart}

\usepackage{amsthm,amsrefs,amsfonts}
\usepackage{mathtools}
\usepackage{tikz-cd}
\usepackage{verbatim}
\usepackage{url}
\usepackage{hyperref}
\usepackage{amsmath}
\usepackage{amssymb}
\usepackage{ulem}
\usepackage[all]{xy}
\usepackage{tikz}
\usepackage{tikz-3dplot}

\usepackage{fancyhdr}

\newcommand\numberthis{\addtocounter{equation}{1}\tag{\theequation}}

\tikzcdset{ampersand replacement=\&}

\newtheorem{theorem}{Theorem}[section]
\newtheorem{proposition}[theorem]{Proposition}
\newtheorem{corollary}[theorem]{Corollary}
\newtheorem{lemma}[theorem]{Lemma}
\newtheorem{conjecture}{Conjecture}

\theoremstyle{remark}
\newtheorem*{remark}{Remark}

\theoremstyle{definition}
\newtheorem{definition}{Definition}[section]
\newtheorem{example}{Example}[section]
\newtheorem*{Acknowledgment}{Acknowledgment}

\newcommand{\RN}[1]{%
  \textup{\uppercase\expandafter{\romannumeral#1}}%
}
\newcommand{\colim}{\text{colim}}

\newcommand{\Z}{\mathbb{Z}}

\newcommand{\R}{\mathbb{R}}

\newcommand{\ho}{\mathrm{ho}}
\newcommand{\Loop}{\mathcal{L}}

\newcommand{\Fix}{\text{Fix}}

\newcommand{\fr}{\mathrm{fr}}

\begin{document}
\title[The Klein and Williams Conjecture for the Fixed Point Problem]{On the Klein and Williams Conjecture for the Equivariant Fixed Point Problem}
\author{Ba\c{s}ak K\"u\c{c}\"uk}
\address{Mathematisches Institut, Georg-August-Universit\"at G\"ottingen, Bunsenstraße 3-5, 37073 Göttingen, Germany.}
\email{basak.kucuk@mathematik.uni-goettingen.de}
\maketitle

\begin{abstract}
Klein and Williams developed an obstruction theory for the homotopical equivariant fixed point problem, which asks whether an equivariant map can be deformed, through an equivariant homotopy, into another map with no fixed points \cite[Theorem H]{KW2}. An alternative approach to this problem was given by Fadell and Wong \cite{FW88} using a collection of Nielsen numbers. It remained an open question, stated as a conjecture in \cite{KW2}, whether these Nielsen numbers could be computed from the Klein-Williams invariant. We resolve this conjecture by providing an explicit decomposition of the Klein-Williams invariant under the tom Dieck splitting. Furthermore, we apply these results to the periodic point problem.  
\end{abstract}

%\tableofcontents

\section{Introduction}

The \textit{homotopical fixed point problem} asks whether a given self-map on a topological space can be deformed into a fixed-point-free map. The classical obstruction to this problem for compact ENR spaces is the \textit{Lefschetz number}, as stated in the famous Lefschetz fixed point theorem: if a self-map $f$ on a compact ENR $X$ has no fixed points, then its Lefschetz number $L(f)$ must be zero. Since the Lefschetz number is defined using homology, it remains invariant under homotopy. 

When $X$ is a simply-connected space of dimension at least three, the Lefschetz number is a complete invariant: $f$ is homotopic to a fixed-point-free map if and only if $L(f) = 0$. However, if $X$ is not simply connected, the converse of the Lefschetz theorem does not necessarily hold—see \cite{brown.converse.fixpt} for details.

A more refined invariant, called the \textit{Nielsen number} $N(f)$ (introduced by Nielsen \cite{Nielsen1921}), provides a stronger result: for a self-map $f$ on a compact ENR $X$ with $\dim X \geq 3$, $f$ can be deformed into a fixed-point-free map if and only if $N(f) = 0$. More precisely, the Nielsen number gives a lower bound on the number of fixed points of $f$, up to homotopy (Wecken \cite{Wecken1942}). The dimension condition is necessary, as Jiang \cite{jiang2dim} showed that for any two-dimensional connected manifold with negative Euler characteristic, there exists an onto self-map that is not homotopic to a fixed-point-free map, despite having $N(f) = 0$.

Classically, the Nielsen number is defined geometrically by counting essential fixed-point classes. However, it can also be computed algebraically from the \textit{Reidemeister trace} \cite{Reidemeister1936AutomorphismenVH, Wecken1941}, which is also a complete invariant for the homotopical fixed-point problem.

This paper focuses on an equivariant generalization of the problem. Several equivariant versions of the Lefschetz fixed point theorem and generalized equivariant Lefschetz invariants have been developed using different approaches. One such approach, by \cite{FW88}, extends the classical case to the equivariant setting. Their results show that for a compact Lie group $G$, a collection of Nielsen numbers $N(f^H)$ of the induced maps on fixed-point sets $M^H$, for $H \leq G$, provides a complete invariant for determining whether a $G$-map $f$ can be deformed equivariantly into a fixed-point-free map.

\begin{theorem}{\cite[Fadell, Wong]{FW88}}\label{FWresult}
Let $G$ be a compact Lie group and $M$ a compact, smooth $G$-manifold. Let $(H_1), \ldots, (H_k)$ be an admissible ordering of the isotropy types of $M$, and $M_i=\{ x \in M:(G_x)=(H_j), j \leq i\}$ be the associated filtration. Suppose that for each $i, 1\leq i \leq k$, the Weyl group $WH_i=N H_i / H_i$ is finite, $\dim M^{H_i} \geq 3$, and $\dim (M_{i-1}\cap M^{H_i}) \leq \dim M^{H_i}-2$. Then, a $G$-self-map $f$ on $M$ is $G$-homotopic to a fixed point free $G$-map if and only if the Nielsen number $N(f^{H_i})=0$ for each $i$.
\end{theorem}

Later, Wong \cite{WongNielsenNumber} defined the equivariant Nielsen number for compact Lie groups, which provides the minimal number of fixed orbits and fixed points of an equivariant map.  

A more general and modern version of this invariant for infinite discrete groups was introduced by Weber \cite{weber07}. Weber’s definition is structurally more refined, as it allows one to extract equivariant Nielsen invariants directly from the generalized equivariant Lefschetz invariant, which she previously developed in \cite{Weber06}.  

Another significant approach was introduced by Klein and Williams \cite{KW2}, who constructed a homotopical invariant for closed, smooth $G$-manifolds, where $G$ is a finite group. 

\begin{theorem}{\cite{KW2}}\label{KWfix}Let $f: M \rightarrow M$ be a $G$-map on a closed, smooth $G$-manifold $M$. Then, there exists an invariant

$$
\ell_G(f) \in \Omega_0^{G, \fr}(\Loop_f M),
$$

which vanishes if $f$ is $G$-equivariantly homotopic to a fixed-point-free map. Here,

$$
\Loop_f M = \{ \lambda \colon [0,1] \to M \mid f(\lambda(0)) = \lambda(1) \}
$$

denotes the space of paths twisted by $f$, and $\Omega_0^{G, \fr}(\Loop_f M)$ is the $G$-equivariant framed bordism group of $\Loop_f M$. Conversely, assume that $\ell_G(f) = 0$. Suppose the following conditions hold:

\begin{itemize}
\item $\dim M^H \geq 3$ for all conjugacy classes of subgroups $(H)$ such that the subgroup $H \subset G$ appears as an isotropy group in $M$, and
\item $\dim M^H \leq \dim M^K - 2$ for all conjugacy classes $(H), (K)$ with proper subgroup inclusions $K \subset H$, where $H$ and $K$ are isotropy subgroups of $M$.
\end{itemize}

Then $f$ is $G$-equivariantly homotopic to a fixed-point-free map.
\end{theorem} 

The dimension conditions in the hypothesis of Theorem \ref{KWfix} are crucial and cannot be omitted, as demonstrated by Ferrario \cite[Example 5.1]{Ferrario99}. Moreover, these conditions are equivalent to the standard gap condition, which appears as the dimension hypothesis in Theorem \ref{FWresult}.  

An equivariant framed bordism group $\Omega_n^{G,\fr}(X)$ is defined on the set of equivalence classes of triples $(M, g, \phi)$ under the relation of bordism. Here, $(M, g, \phi)$ consists of an equivariant (tangential) stable framed $n$-manifold $M$, an equivariant map $g: M \to X$, and a $G$-homotopy class of $G$-bundle trivializations $\phi$ of the tangent bundle $\tau_M$, given by  
\[
\phi: \tau_M \oplus (\mathbb{R}^k \times M) \cong_G (\mathbb{R}^{n+k} \times M),
\]
where $\mathbb{R}^k$ is considered as a trivial $G$-representation.  

This definition corresponds to the naive equivariant bordism and does not incorporate $RO(G)$-graded bordism groups. For a more general treatment of equivariant framed bordism in $RO(G)$-grading and its Pontryagin-Thom construction into equivariant stable homotopy groups, we refer to \cite{equivframed.kosniowski}.  

In this paper, we focus on the $0$-th equivariant framed bordism group, and in this case, the Pontryagin-Thom construction yields the following isomorphism:  
\[
\Omega_0^{G,\fr}(X) \cong \pi_0^{G,st}(X) := \mathop\colim_{n\to \infty} [S^n, S^n \wedge X_+]_G.
\]  

\begin{remark}\label{tomDieck} The tom Dieck splitting \cite{Dieck1987} for this equivariant framed bordism group decomposes it into a direct sum of non-equivariant framed bordism groups, indexed by the conjugacy classes of subgroups of $G$.

\begin{align*}
\Omega_0^{G,\fr}(\Loop_f M) \cong \bigoplus_{(H)} \Omega_0^{\fr} (EWH \times_{WH} \Loop_f M^H),
\end{align*}
where $WH$ is Weyl group defined by $WH:= N_G(H)/H$.
\end{remark}

The conjecture stated by Klein and Williams is the main motivation of this paper. To this end, we first analyze the Klein-Williams invariant $\ell_G(f)$ in detail and then provide an affirmative answer to the following conjecture.

\begin{conjecture}{\cite{KW2}}
    The Nielsen numbers $N(f^H)$ of the induced maps on the fixed sets for each conjugacy class of subgroups $H$ of $G$ can be computed from the projection of $\ell_G(f)$ onto the corresponding summand in the tom Dieck splitting.
\end{conjecture}

The structure of the paper is as follows. In Section \ref{sec2}, we provide the necessary background to construct the Klein-Williams invariant. In Section \ref{sec3}, we present a detailed construction of this invariant, which is essential for establishing its decomposition under the tom Dieck splitting.

In Section \ref{sec4}, we prove Theorem \ref{splittingresult}, which provides a splitting formula for the Klein–Williams invariant. After establishing this splitting, we introduce the geometric Reidemeister trace $R(f)$ and show in Theorem \ref{ell(f)} that the non-equivariant invariant $\ell(f)$ coincides with $R(f)$. This identification leads to Theorem \ref{ident3}, which relates the components of the Klein–Williams invariant to the reduced Reidemeister traces.

We conclude Section \ref{sec4} with Proposition \ref{fix.point.index}, which enables the proof of Theorem \ref{quotientR(f)}. This result resolves Conjecture \ref{conj} by showing that the Klein–Williams invariant $\ell_G(f)$ vanishes if and only if $N(f^H) = 0$ for all conjugacy classes $(H)$ of subgroups.

In Section \ref{sec5}, we apply this fixed-point theory to a periodic point problem. The obstruction theory for deforming a map to eliminate periodic points of period $n$ is developed in \cite[Section 11]{KW2}. Corollary \ref{proof.conj} shows that the obstruction $\ell_n(f)$ is a complete invariant and contains the same amount of information as the Nielsen numbers $N(f^k)$ for all divisors $k$ of $n$. We conclude the paper with an explicit Example \ref{ex.periodic}, demonstrating that the number of nonzero terms in the projections of the Klein-Williams invariant does not always match the corresponding Nielsen numbers, even though both vanish simultaneously.

\begin{Acknowledgment}
This work is part of the author's PhD project under the supervision of Thomas Schick. The author expresses sincere gratitude to Thomas Schick for his invaluable guidance, continuous support, and many fruitful discussions throughout the development of this work. This work was supported by the German Academic Exchange Service (DAAD) through the Graduate School Scholarship Programme.
\end{Acknowledgment}

\section{Preliminaries}\label{sec2}

In this section, we provide definitions and results in parametrized homotopy theory, which are crucial for the formulation of the invariant introduced by Klein and Williams. For detailed explanations, refer to \cite{may2004parametrized} and \cite{lewis2006equivariant}.

\subsection*{Fiberwise $G$-Spaces.} 

We will adopt the same notation used by Klein and Williams \cite{KW2}. Let $B$ be a $G$-space, and $T(B;G)$ is the category of $G$-spaces over $B$. That is, its objects are $G$-spaces $X$ equipped with a structure $G$-map $p: X \to B$. Morphisms are just $G$-maps which are compatible with the structure maps. If objects $X$ have extra structure map $ s \colon B \to X$ such that $p \circ s = \text{id}_B$, and maps between these objects are compatible with both structure maps, then we denote this ``retractive" version of category of $G$-spaces over $B$ as $R(B;G)$.

Both $T(B;G)$ and $R(B;G)$ are a model category, and the following definitions of \textit{weak equivalence}, \textit{fibration} and \textit{cofibration} maps are the same in both categories. Details of the model categories can be found on \cite{hovey2007model}.

A map $f: X \to Y$ is called \textit{weak equivalence} if for every subgroup $H$ of $G$, the induced map of fixed points $f^H:X^H \to Y^H$ is a weak homotopy equivalence. Moreover, $f$ is called \textit{fibration} if the induced map $f^H$ is a Serre fibration for every subgroup $H$. It is called \textit{cofibration} if there is a relative $G$-cell complex $(Z,X)$ such that $Y$ is a retract of $Z$ relative to $X$.

\begin{definition}
Let $E$ be an object in $T(B;G)$. The \textit{unreduced fiberwise suspension} of $E$ over $B$ is object $S_B E$ in $T(B;G)$ given by
\begin{align*}
    S_B E := B \times \partial [0,1] \bigcup_{E \times \partial [0,1]} E \times [0,1].
\end{align*}
\end{definition}

\begin{definition}
Let $E$ be an object in $R(B;G)$. The \textit{reduced fiberwise suspension} of $E$ over $B$ is object $\Sigma_B E$ in $T(B;G)$ given by the pushout of the diagram
\begin{align*}
    B \leftarrow S_B B \rightarrow S_B E
\end{align*}
where $S_B B \to S_B E$ is just the application of $S_B$ functor to the structure map $B \to E$.
\end{definition}

After defining the parametrized smash product, we will explore its relationship with the reduced suspension in the fiberwise setting. This connection will be useful in our later discussions.

\begin{definition}
Let $X$ and $Y$ be objects in $R(B;G)$. The \textit{internal smash product} or \textit{fiberwise smash product} $X \wedge_B Y$ of $X$ and $Y$ is the object given by the following pushout of the diagram
\begin{align*}
B \leftarrow X \cup_B Y \rightarrow X \times_B Y
\end{align*}
where $X\times_B Y$ is the fiber product of $X$ and $Y$.
\end{definition}

\begin{proposition} \label{redsus}
Let $E$ be an object in $R(B;G)$. Then, $n$-th iterated reduced fiberwise suspension of $E$ is homeomorphic to the following smash product.
\begin{align*}
  \Sigma_B^n E  \cong (S^n \times B) \wedge_B E 
\end{align*}    
\end{proposition}

\begin{proof}
It is clear that $S^n \times B$ is an object in $R(B;G)$ with structure maps given by
\[
\operatorname{pr}_2 : S^n \times B \to B, (b,x) \mapsto b  \: \text{   and  } s: B \to S^n \times B,  b \mapsto (b,x_0)
\]
for some $x_0 \in S^n$.

We consider the fiber of $(S^n \times B) \wedge_B E$. By the definition of fiberwise smash product, we have
\[
(S^n \times B) \wedge_B E= (S^n \times B) \times_B E \cup_{(S^n \times B) \cup_B E} B.
\]
Moreover, the inverse of the structure map $p$ of $(S^n \times B) \wedge_B E$ at a point $b \in B$ is the following
\[
(S^n \times p^{-1}(b))  \cup_{(S^n \times \{b\}) \cup_{\{b\}} p^{-1}(b)} \{b\} \cong (S^n \times p^{-1}(b)) /S^n  \cup_{\{b\}} p^{-1}(b).
\]

This identification yields the smash product $S^n \wedge p^{-1}(b)$. Since 
$$
S^n \wedge p^{-1}(b) \cong \Sigma^n p^{-1}(b),
$$

two objects $(S^n \times B) \wedge_B E$ and $\Sigma^n_B E$ are homeomorphic fiberwise $G$-spaces.
\end{proof}

\begin{definition}
\textit{A naive parametrized $G$-spectrum} $\mathcal{E}$ consists of a family of objects $\mathcal{E}_n \in R(B;G)$ equipped with maps $\Sigma_B \mathcal{E}_n \to \mathcal{E}_{n+1}$.
\end{definition}

\begin{example}
Let $X \in R(B;G)$ be an object. Then, the \textit{naive parametrized suspension spectrum} $\Sigma_B^\infty X$ has $n$-th object $\Sigma^n_BX$, which is the $n$-th iterated reduced fiberwise suspension of $X$.
\end{example}

\begin{definition}
Let $X$ be an object in $T(B;G)$. The \textit{cohomology} of $X$ with coefficients in $\mathcal{E}$ is the group given by
\begin{align*}
    H^k_G(X;\mathcal{E}) := \mathop\colim_{n \to \infty} [\Sigma_B^n X^+, \mathcal{E}_{n+k}]_{R(B;G)}
\end{align*}
where $X^+=X \sqcup B$.
\end{definition}

Similarly, one can define naive parametrized equivariant homology as follows.

\begin{definition}
Let $X$ be an object in $T(B;G)$. The \textit{homology} of $X$ with coefficients in $\mathcal{E}$ is the group given by
\begin{align*}
    H_k^G(X;\mathcal{E}) := \mathop\colim_{n \to \infty} [S^{n+k}, (X^+ \wedge_B \mathcal{E}_{n})]_{R(*;G)}
\end{align*}
\end{definition}

\section{The Klein-Williams Invariant}\label{sec3}

In this section, we provide a construction of the Klein-Williams invariant $\ell_G(f)$, which is essential for understanding its decomposition under the tom Dieck splitting.

Klein and Williams first developed an obstruction theory for a general problem known as the \textit{intersection problem}. The \textit{equivariant intersection problem} asks the following: given a compact $G$-manifold $N$ with a closed submanifold $Q$ and an equivariant map $f \colon P \to N$, where $P$ is a closed $G$-manifold, when can one find an equivariant homotopy from $f$ to a map into the complement of $Q$?  

\[
\begin{tikzcd} 
\& \: \& N-Q \arrow[d, ""]\\
\& P \arrow[ur,dashed,""] \arrow[r,"f"'] \& N
\end{tikzcd}
\]

We will explain how the equivariant intersection problem can be applied to the equivariant fixed-point problem. Let $M$ be a closed, smooth $G$-manifold, where $G$ is a finite group. First, consider the following commutative square of equivariant mapping spaces.  

\begin{equation}\label{infinitycartesian}
\begin{tikzcd}
\& \operatorname{end}^{\star}(M)^{G} \arrow[r, hook] \arrow[d,"\Gamma"] \& \operatorname{end}(M)^{G} \arrow[d,"\Gamma"] \\
\& \operatorname{map}(M, M \times M - \Delta)^{G} \arrow[r, hook] \& \operatorname{map}(M, M \times M)^{G}
\end{tikzcd}
\end{equation}

Here, $\Delta \subset M \times M$ is the diagonal subspace, the action on $M \times M$ is the diagonal $G$-action, $\operatorname{end}(M)^G$ is the space of equivariant self-maps of $M$, and $\operatorname{end}^{\star}(M)^G$ denotes the space of fixed-point-free $G$-maps from $M$ to itself. The horizontal maps are injections, and the vertical maps $\Gamma$ send a map $f$ to $\Gamma_f$, which is the graph of $f$. 

Suppose that this square is $0$-cartesian; that is, the induced map on the set of components of the universal map from $\operatorname{end}^{\star}(M)^G$ to the homotopy pullback of the square is surjective. Then, solving the equivariant fixed-point problem reduces to solving the equivariant intersection problem for the following diagram:  

\begin{equation}\label{graphmap}
\begin{tikzcd}
\& \: \& M \times M - \Delta \arrow[d, ""]\\
\& M \arrow[ur,dashed,""] \arrow[r,"\Gamma_f"'] \& M \times M
\end{tikzcd}
\end{equation}

The square \eqref{infinitycartesian} is indeed $\infty$-cartesian, meaning that it is a homotopy pullback. This is proven in \cite[Lemma 10.1]{KW2}. As a result, by applying the diagram \eqref{graphmap} to the obstruction theory for the equivariant intersection problem, we obtain the Klein-Williams invariant $\ell_G(f)$.

To achieve this, consider the inclusion map $M \times M - \Delta \to M \times M$ and convert this map into a fibration:  
\begin{align*}
    M \times M - \Delta \simeq E \to M \times M
\end{align*}

Explicitly, $E$ is the pullback of path fibration along the inclusion map, and the fibration map from $E$ to $M\times M$ is given by $((m_1,m_2),\lambda) \mapsto \lambda(1)$. Finding a lift such that the diagram \eqref{graphmap} homotopy commutes is equivalent to finding an equivariant section for the fibration $\Gamma_f^*E \to M$.

The following theorem shows that the class $[s] \in H^0_G(M; \Sigma_M^\infty S_M \Gamma_f^*E)$ obtained by the map
\begin{align*}
    s:= s_- \sqcup s_+ \colon M^+ \to S_M \Gamma_f^*E
\end{align*}
gives the obstruction to finding an equivariant section.

\begin{theorem}{\cite[3.1]{KW2}}
Let $E$ be a fibrant object in the category $T(M;G)$. If $E \to M$ admits an equivariant section, then $[s]$ is trivial. Conversely, If $[s]$ is trivial, $E^H$ is $r_H$-connected, and $m_H \leq 2r_H +1$, where $m_H=\dim M^H$, then $E \to B$ admits an equivariant section.
\end{theorem}

The following isomorphism between the parametrized equivariant cohomology group and the equivariant framed bordism group is known; indeed, a more general version appears in the proof of \cite[Theorem B]{KW2}. This isomorphism defines the Klein-Williams invariant. 

\begin{theorem} \label{iso}
There exists isomorphism between groups
\begin{align*}
H^0_G(M;\Sigma_M^\infty S_M \Gamma_f^*E) \cong \Omega_0^{G,\fr}(\Loop_f M)
\end{align*}
\end{theorem}

\begin{definition}
The \textit{Klein-Williams} invariant $\ell_G(f)$ is defined as the image of $[s]$ under the isomorphism in Theorem \ref{iso}:

\begin{align*} 
H^0_G(M;\Sigma_M^\infty S_M \Gamma_f^*E) & \xrightarrow{\cong} \Omega_0^{G,\fr}(\Loop_f M)\\
[s] & \mapsto \ell_G(f)
%\numberthis \label{ellG}
\end{align*}
\end{definition}

Before proving Theorem \ref{iso}, we first establish a homotopy equivalence between $S_M \Gamma_f^* E$ and $S^{\tau_M} \wedge_M \Loop_f M$ in $T(B;G)$, where $S^{\tau_M}$ denotes the fiberwise one-point compactification of the tangent bundle $\tau_M$ of $M$.

\begin{lemma}\label{w.e.}
There is an equivariant weak equivalence between following spaces.
\begin{align*}
S_M \Gamma_f^* E \simeq S^{\tau_M} \wedge_M (\Loop_f M)^+
\end{align*}
\end{lemma}

\begin{proof}
Let $\nu$ be a normal bundle of $\Delta$ in $M \times M$. We can identify disk bundle $D(\nu)$ of $\nu$ with equivariant tubular neighborhood of $\Delta$. Then, 
\[
M \times M \cong D(\nu) \cup_{S(\nu)} (M\times M - \text{int}D(\nu))
\] 
where $S(\nu)$ is sphere bundle of $\nu$.
Consider the following equivariant commutative diagram:
\[
\begin{tikzcd}
\& S(\nu) \arrow[r,""] \arrow[d,""] \& M \times M - \text{int}D(\nu) \arrow[d, ""] \arrow[r,""] \& M\times M \arrow[d,""]\\
\& D(\nu) \arrow[r,""] \& M \times M \arrow[r,""] \& (M\times M) \cup_{M \times M - \text{int}D(\nu)} (M \times M)
\end{tikzcd}
\]

Since the first and second squares are equivariant pushout, the big rectangular must be equivariant pushout. This gives that

\begin{align*}
D(\nu) \cup_{S(\nu)} (M\times M) \cong  (M\times M) \cup_{M \times M - \text{int}D(\nu)} (M \times M).
\end{align*}

Observe that the right hand side is homotopy equivalent with the fiberwise suspension $S_{M \times M}(M \times M - \Delta)$. Moreover, the left hand side of the homeomorphism is homotopy equivalent to $S^{\nu} \cup_\Delta (M\times M)$, where $S^{\nu}$ is fiberwise one-point compactification of $\nu$. Recall that $E \simeq M \times M - \Delta$. Therefore, we have a weak equivalence between based spaces over $M\times M$:
\begin{align*}
S_{M \times M}E \simeq S^{\nu} \cup_\Delta (M\times M)    
\end{align*}

To complete the proof, we pullback these homotopy equivalent based spaces along the graph map $\Gamma_f$. Observe that $S_{M \times M}E$ is a fibrant object in the category $R(M \times M; G)$, and pulling back along $\Gamma_f$ yields the object $S_M \Gamma_f^* E$ in $R(M;G)$.

For the right-hand side, we must proceed more carefully, since the corresponding space is not fibrant in $R(M \times M; G)$. To remedy this, we apply the same fibrant replacement trick as before. Specifically, we consider the pullback of the path fibration along the map
$$
S^{\nu} \cup_\Delta (M \times M) \to M \times M,
$$
which gives rise to the space
$$
(S^{\nu} \cup_\Delta (M \times M)) \times_{M \times M} (M \times M)^I.
$$
The projection map for this fibration is defined as
$$
\begin{aligned}
(S^{\nu} \cup_\Delta (M \times M)) \times_{M \times M} (M \times M)^I &\longrightarrow M \times M \\
(x, \lambda) &\longmapsto \lambda(1).
\end{aligned}
$$
Pulling this fibrant object back along the graph map $\Gamma_f$ gives the following space:
\begin{equation} \label{long}
(S^{\nu} \cup_\Delta (M\times M)) \times_{M\times M} (M \times M)^I  \times_{M \times M} M.  
\end{equation}
It remains to identify this space up to homotopy with $S^{\tau_M} \wedge_M \Loop_f M$. First, observe that
\[
(S^{\nu} \times_{M\times M} (M \times M)^I  \times_{M \times M} M) \cong S^{\tau_M} \times_M \Loop_f M,
\]
by unrevealing the definitions of pullbacks. Similarly, we have
\[\Delta \times_{M\times M} (M \times M)^I  \times_{M \times M} M \cong \Loop_f M.\]
After distributing the fiber products in \eqref{long} and writing everything out explicitly, we observe that \eqref{long} is homeomorphic to the following:
\begin{align*}
(S^{\tau_M} \times_M \Loop_f M) \bigcup_{\Loop_f M} (M \times M)^I \times_{M \times M} M
\end{align*}
Also, note that there exists a homotopy equivalence
$$
h \colon (M \times M)^I \times_{M \times M} M \to M,
$$
given by $(\lambda, m) \mapsto m$. Its homotopy inverse is defined by
$$
m \mapsto (\text{const}_{(m, f(m))}, m),
$$
where $\text{const}_{(m, f(m))}$ denotes the constant path from $m$ to $f(m)$. Using this homotopy equivalence, we deduce that the space in \eqref{long} is homotopy equivalent to
$$
(S^{\tau_M} \times_M \Loop_f M) \cup_{\Loop_f M} M.
$$
This space is homeomorphic to the fiberwise half-smash product over $M$,
$$
S^{\tau_M} \wedge_M (\Loop_f M)^+.
$$
\end{proof}

\begin{proof}[Proof of Theorem \ref{iso}]

By Proposition \ref{redsus}, we have that

\[
\Sigma_M^nM^+\cong (S^n\times M) \wedge_M M^+.
\]

Using the definition of the fiberwise smash product, this space is homeomorphic to $S^n\times M$. Therefore, the parametrized equivariant cohomology of $M$ over itself is given by the following:

\begin{align*}
H^0_G(M;\Sigma_M^\infty S_M \Gamma_f^*E) & := \colim_{n\to \infty}[\Sigma^n_M M^+,\Sigma_M^n S_M \Gamma_f^*E]_{R(M;G)}\\
% & \cong \colim_{n\to \infty}[M^+,\Omega_M^n \Sigma_M^n S_M \Gamma_f^*E]_{R(M;G)}\\
& = \colim_{n\to \infty}[S^n \times M, \Sigma^n_M S_M \Gamma_f^*E]_{R(M;G)}\\
& \cong \colim_{n\to \infty}[S^n,\sec(\Sigma^n_M S_M \Gamma_f^*E)]_{R(*;G)}
\end{align*}

The last statement follows from the fact that the functor 
\[ 
h^*: R(*;G) \to R(M;G) 
\]
given by $Y \mapsto Y \times M$, has a right adjoint $h_*$. The functor 
\[
h_*: R(M;G) \to R(*;G) 
\]
is defined by $X \mapsto \sec(X)$, where $\sec(X)$ denotes the space of sections of the fiberwise space $X$ over $M$. Moreover, the functor $h^*$ admits a left adjoint $h_\#$, which is given by

\begin{align*}
h_\# : R(M;G) & \to R(*;G)\\
X & \mapsto X/M
\end{align*}

By rewriting the space $\sec(\Sigma^n_M S_M \Gamma_f^*E)$, one can express it as 
\[ 
h_*(h^*(S^n) \wedge_M S_M \Gamma_f^*E) 
\]
using the identity established in Proposition \ref{redsus}. Then, by Lemma \ref{w.e.}, we have

\[
h_*(h^*(S^n) \wedge_M S_M \Gamma_f^*E) \simeq h_*(h^*(S^n) \wedge_M S^{\tau_M} \wedge_M (\Loop_f M)^+).
\]

Thus, we obtained that

\[
H^0_G(M;\Sigma_M^\infty S_M \Gamma_f^*E) = \colim_{n\to \infty}[S^n,h_*(h^*(S^n) \wedge_M S^{\tau_M} \wedge_M (\Loop_f M)^+)]_{R(*;G)}.
\]

Note that there exits a $G$-spectrum $S^{-\tau_M}$ over $M$ indexed on the trivial universe, which is defined by the fiberwise functional dual of $S^{\tau_M}$. That is, 

\[
S^{-\tau_M} \wedge_M S^{\tau_M}   \simeq \Sigma_M^\infty M^+.
\]

We can understand $S^{-\tau_M}$ in the unstable setting as follows. Let $M$ be embedded in $\mathbb{R}^m$, and let $\nu$ denote the normal bundle of $M$ in $\mathbb{R}^m$. Then, the fiberwise one-point compactification $S^\nu$ of the normal bundle represents $S^{-\tau_M}$, up to a degree shift by $m$. That is,

$$
S^{-\tau_M} \simeq S^{\nu - \varepsilon^m},
$$

where $\varepsilon^m$ denotes the trivial bundle over $M$ of rank $m$. For a parametrized spectrum $\mathcal{E}$ over $M$, we define $S^{-\tau_M}\wedge_M \mathcal{E}$, the smash product between two spectra $S^{-\tau_M}$ and $\mathcal{E}$ as follows. It has $n$-th space 

\[
(S^{-\tau_M}\wedge_M \mathcal{E})_n := \Omega^m_M(S^\nu \wedge_M \mathcal{E}_n),
\]

where $\Omega_M^m(S^\nu \wedge_M \mathcal{E}_n)$ is the fiberwise loop space, which consists of tuples $(m,\lambda)$ such that $m \in M$ and $\lambda: S^m \to p^{-1}(m)$ is a based map, where $p$ is the structure map of $S^\nu \wedge_M \mathcal{E}_n$. Note that $(\Sigma_M,\Omega_M)$ is an adjoint functor pair in the category $T(M;G)$.

Now, we will use the special case of \cite[Theorem 4.9]{HU}, which states that for any $G$-spectra $\mathcal{E}$, we have an equivalence of $G$-spectra:
\begin{align*}
    h_\# (S^{-\tau_M} \wedge_M \mathcal{E} ) \simeq h_* \mathcal{E},
\end{align*}

where $h_\#$ and $h_*$ are covariant functors between the categories of parametrized $G$-spectra over $M$ and a single point $*$, which are defined level-wise of spaces.

Now, let $\mathcal{E}$ be a parametrized spectrum over $M$, which has $n$-th space 

\[
h^*(S^n) \wedge_M S^{\tau_M} \wedge_M (\Loop_f M)^+.
\]

Then, by applying \cite[Theorem 4.9]{HU}, we obtain that $h_*(\mathcal{E})$ is weak equivalent to $h_\#(S^{-\tau_M} \wedge_M \mathcal{E)}$, which has $n$-th space

\[
h_\#(\Omega_M^m (S^{\nu} \wedge_M h^*(S^n) \wedge_M S^{\tau_M} \wedge_M (\Loop_f M)^+)).
\]
By unrevealing the definitions, and using the identity $S^\nu\wedge_MS^{\tau_M} \cong S^{\nu \oplus \tau_M}$ and \cite[Lemma 4.7]{HU}, we obtain that $n$-th space of $h_\#(S^{-\tau_M} \wedge_M \mathcal{E)}$ is equivalent to the following.
\begin{align*}
& h_\#(\Omega_M^m (S^{v\oplus\tau_M}\wedge_Mh_*(S^n)\wedge_M(\Loop_f M)^+))\\
&\cong h_\#(\Omega_M^m ((S^{m}\times M)\wedge_Mh_*(S^n)\wedge_M(\Loop_f M)^+))\\
&\cong h_\#(\Omega_M^m \Sigma_M^m(h_*(S^n)\wedge_M(\Loop_f M)^+))\\
%&\cong \Omega^m \Sigma^m h_\# (h_*(S^n)\wedge_M \Loop_f M^+)\\
&\cong \Omega^m \Sigma^m (S^n \wedge h_\# ((\Loop_f M)^+))\\
&\cong \Omega^m \Sigma^m (S^n \wedge (\Loop_f M)_+).
\end{align*}
As a result, we obtain that
\begin{align*}
H^0_G(M;\Sigma_M^\infty S_M \Gamma_f^*E) & \cong \colim_{n\to \infty}[S^{n+m},  S^{n+m} \wedge (\Loop_f M)_+]_{R(*;G)}\\
& = \pi^{G,st}_0((\Loop_f M)_+)\\
& \cong  \Omega_0^{G,\fr}(\Loop_f M).
\end{align*}\end{proof}

\section{Conjecture of the Klein-Williams}\label{sec4}

In this section, we present our results that affirmatively answer Conjecture \ref{conj}. We begin by decomposing the equivariant framed bordism group into non-equivariant framed bordism groups using tom Dieck splitting (see Remark \ref{tomDieck}).  This decomposition enables us to study the invariant $\ell_G(f)$ componentwise, where each component lies in the group 
\[
\Omega_0^{\fr} (EWH \times_{WH} \Loop_f M^H)
\]

where $WH$ denotes the Weyl group. The zeroth framed bordism groups of a space are well-known and correspond to the free abelian group generated by its path components. Furthermore, leveraging Lemma \ref{borel}, which provides a key property of the Borel construction, one can conclude the following:

\begin{align*}
\Omega_0^{G,\fr}(\Loop_f M) & \cong \bigoplus_{(H)} \Z [\pi_0(EWH \times_{WH} \Loop_f M^H)]\\
&\cong \bigoplus_{(H)} \Z [\pi_0(\Loop_f M^H)/WH]\\
& \cong \bigoplus_{(H)} \Z [\pi_0(\Loop_f M^H)]/ WH\\
& \cong \bigoplus_{(H)} \Omega_0^{\fr} (\Loop_f M^H)/ WH 
\numberthis \label{tomDsplitting}
\end{align*}

\begin{lemma}{\label{borel}}
There is a one-to-one correspondence between path components of Borel construction $EG \times_{G} X$ and orbit space of path components of $G$-space $X$.
\end{lemma}

\begin{proof}
We will consider the canonical map between these sets and we will show this map is one-to-one.
\begin{align*}
    \pi_0(EG \times_G X) &\rightarrow \pi_0 (X)/G\\
    [(e,x)] &\mapsto [x]
\end{align*}
Let two elements $(e,x)$ and $(e',x')$ belong to the Borel construction, and assume they are in the same path component. Clearly, this implies that $x$ and $x'$ lie in the same component of $X$, this implies that the map is well-defined.

Now, suppose that $[x] = [x']$ in $\pi_0(X)/G$, but $x$ and $x'$ are not in the same component of $X$. Then there exists some $g \in G$ such that $g \cdot x' = x$. It suffices to show that $[(e,x)] = [(e',g \cdot x)]$ for all $e, e' \in EG$.

To establish this, consider $e' = g \cdot f$ for some $f \in EG$. Then, $[(e',g \cdot x)] = [(f,x)]$. Since $EG$ is connected, there exists a path in $EG$ between any two points, Therefore, $[(e,x)] = [(f,x)]$ in $\pi_0(EG\times_GX)$. Thus, $[(e,x)] =[(e',g \cdot x)]$. This proves injectivity. Surjectivity follows directly from the definition of the quotient.
\end{proof}

The following result provides the decomposition of the Klein-Williams invariant $\ell_G(f)$. Since $\ell_G(f)$ is defined as the image of $[s] \in H^0_G(M;\Sigma_M^\infty S_M \Gamma_f^*E)$, we first establish the decomposition of $H^0_G(M;\Sigma_M^\infty S_M \Gamma_f^*E)$ below.  

\begin{theorem} \label{splittingresult}  
There exists a splitting of the zeroth equivariant cohomology of $M$ with coefficients in a fiberwise reduced suspension:  
\begin{align*}  
H^0_G(M;\Sigma_M^\infty S_M \Gamma_f^*E) & \xrightarrow{\cong} \bigoplus_{(H)} H^0(M^H;\Sigma_{M^H}^\infty S_{M^H} \Gamma_{f^H}^*E^H)/WH\\  
[s] & \mapsto \bigoplus_{(H)} \overline{[s^H]}.  
\end{align*}  
Consequently, the invariant $\ell_G(f)$ decomposes as follows under the tom Dieck splitting:  
\begin{align*}  
\Omega_0^{G,\fr}(\Loop_f M) & \xrightarrow{\cong} \bigoplus_{(H)} \Omega_0^{\fr} (\Loop_f M^H)/ WH \\  
\ell_G(f) & \mapsto \bigoplus_{(H)} \overline{\ell(f^H)},  
\end{align*}  
where each $\ell(f^H)$ represents the obstruction to the fixed point problem in the non-equivariant case, and $\overline{\ell(f^H)}$ is the induced element under the quotient by the Weyl group $WH$.  
\end{theorem}  

The decomposition of $H^0_G(M;\Sigma_M^\infty S_M \Gamma_f^*E)$ is straightforward; we only need to determine the image of $\ell_G(f)$. To this end, we will use the following results.  

\begin{theorem}{\cite{KW1}}\label{KWindex}
There exists a bijection
\begin{align*}
\pi_0 ( \sec^{st} (S_M \Gamma_f^*E)) & \xrightarrow{\cong} \Omega_0^{\fr}(\Loop_f M)\\
[s_+] & \mapsto [\Gamma_f \pitchfork i_\Delta]
\end{align*}
where $\text{sec}^{st} (S_M \Gamma_f^*E)$ is the space of sections of $\Omega_M^\infty \Sigma_M^\infty S_M \Gamma_f^*E$ over $M$.   
\end{theorem}

The notation $[\Gamma_f \pitchfork i_\Delta]$ defined by Hatcher and Quinn \cite{HatcherQuinn}, and it carries the information that $D:=\Gamma_f^{-1}(\Delta)$ equipped with map $D \rightarrow \Loop_fM$, given by $x \mapsto const_x$. Note that it is provided that $\Gamma_f$ and $i_\Delta$ are intersecting transversally. It is clear that when $f$ is identity, the Hatcher-Quinn invariant gives the same information with Euler characteristic of $M$ since it is just the self-intersection number of $M$ (It actually counts oriented intersection of the image of $\Gamma_f$ with $\Delta$ in $M \times M$).

\begin{remark}
    The class $[s_+]$ defines an obstruction theory for the fixed point problem in the non-equivariant case. Therefore, we can identify $[s_+]$ with $\ell(f)$. For further details, see \cite{KW1}.
\end{remark}

The following result shows that $\pi_0 ( \sec^{st} (S_M \Gamma_f^*E))$ is equivalent to the (non-equivariant) cohomology group $H^0(M;\Sigma_M^\infty S_M \Gamma_f^*E)$.

\begin{lemma} \label{secst}
There exits an isomorphism between abelian groups:
\begin{align*}
\pi_0 (\sec^{st} (S_M \Gamma_f^*E)) \xrightarrow{\cong} H^0(M;\Sigma_M^\infty S_M \Gamma_f^*E)
\end{align*}
Moreover, $[s_+]$ maps to $[s]$ under this isomorphism.
\end{lemma}

\begin{proof}
We will use the duality of the functors $h_*$ and $h^*$, which is explained in the Proof of Theorem \ref{iso}.
\begin{align*}
\pi_0 (\sec^{st} (S_M \Gamma_f^*E)) & := \colim_{n \to \infty}[S^0, h_*(\Omega_M^n \Sigma_M^n S_M \Gamma_f^*E)]_{R(*,e)}\\
& \cong \colim_{n \mapsto \infty}[h^*S^0, \Omega_M^n \Sigma_M^n S_M \Gamma_f^*E]_{R(M;e)}\\
& \cong \colim_{n\to \infty}[M^+,\Omega_M^n \Sigma_M^n S_M \Gamma_f^*E]_{R(M;e)}\\
& \cong \colim_{n\to \infty}[\Sigma^n_M M^+,\Sigma_M^n S_M \Gamma_f^*E]_{R(M;e)}\\
& = H^0(M;\Sigma_M^\infty S_M \Gamma_f^*E)
\end{align*}
%\begin{lemma}\begin{align*}  H^0(M;\Sigma_M^\infty S_M \Gamma_f^*E) \cong \pi_0^G ( \sec^{st} (S_M \Gamma_f^*E)), \end{align*}\end{lemma}

Furthermore, $[s_+]$ maps to $[s]$ in $H^0(M;\Sigma_M^\infty S_M \Gamma_f^*E)$, by the definitions of classes since $[s_-]$ is chosen to be based point for $\pi_0 (\sec^{st} (S_M \Gamma_f^*E))$.  
\end{proof}

Now, we present a proof of Theorem \ref{splittingresult}.

\begin{proof}[Proof of Theorem \ref{splittingresult}]
As a result of Lemma \ref{secst}, we obtain the following isomorphism.
\begin{align*}
    H^0(M^H;\Sigma_{M^H}^\infty S_{M^H} \Gamma_{f^H}^*E^H) & \xrightarrow{\cong} \Omega_0^{\fr} (\Loop_f M^H)\\
    [s^H] & \mapsto \ell(f^H)
\end{align*}
This yields the desired decompositions, and the image of $\ell_G(f)$ under this decomposition can be described as follows using the isomorphisms in the square below: $\ell_G(f)$ maps to $[s]$ via the inverse of the left vertical map. Then, the upper horizontal map sends $[s]$ to $\oplus_{(H)} \overline{[s^H]}$, which, by the above isomorphism, corresponds to $\oplus_{(H)} \overline{\ell(f^H)}$.

\[
\begin{tikzcd}
\& H^0_G(M;\Sigma_M^\infty S_M \Gamma_f^*E) \arrow[r] \arrow[d]
\& \bigoplus_{(H)} H^0(M^H;\Sigma_{M^H}^\infty S_{M^H} \Gamma_{f^H}^*E^H)/WH \arrow[d]\\
\& \Omega_0^{G,\fr}(\Loop_f M) \arrow[r] \& \bigoplus_{(H)} \Omega_0^{\fr} (\Loop_f M^H)/ WH
\end{tikzcd}
\]
\end{proof}

The Klein-Williams invariant $\ell_G(f)$ projects onto the invariant $\overline{\ell(f^H)}$, and the relationship between the invariant $\ell(f^H)$ and the Reidemeister trace $R(f^H)$ is well understood in the non-equivariant case. Ponto \cite{PontoBook} established this identification by introducing a generalization of the trace in symmetric monoidal categories to a trace in bicategories. In contrast, we provide a proof specifically for closed, compact, smooth manifolds using traditional methods, as outlined in \cite[Remark 10.2.(1)]{KW1}.  

To ensure a common understanding, we first recall the geometric definition of the Reidemeister trace. The algebraic Reidemeister trace can be found in \cite{husseini82, geoghegan-handgeotop}, which also contain proofs that the algebraic and geometric definitions coincide.

The geometric Reidemeister trace $R(f)$ is defined by the fixed point indices of fixed points of $f$. Since a fixed point index of a self-map $f$ on a manifold $M$ is defined locally, it suffices to consider maps of the form $f\colon U \to \R^n$, where $U \subset \R^n$ is an open subset. Assume that the set of fixed points $\Fix(f)$ of $f$ is compact. Then, the \textit{fixed point index} $i(f,U)$ is defined via the induced map
$$
(\mathrm{id}-f)_*\colon H_n(U,U-\Fix(f))  \to H_n(\R^n, \R^n-\{0\}),
$$
by setting
$$
[S^n]_F  \mapsto (\mathrm{id}-f)_*[S^{n}]_F:=i(f,U) [S^{n}],
$$
where $[S^{n}]_F$ denotes the image of fundamental class $[S^{n}]$ under the map given by 

$$H_{n}(S^{n}) \to H_{n}(S^{n},S^{n} -\Fix(f)) \cong H_{n}(U,U-\Fix(f)).$$

For a detailed discussion on the fixed point index, we refer to \cite{jiangbook, brownfix, DOLD19651}. The geometric Reidemeister trace is defined in terms of the \textit{fixed point classes} of $f$, denoted by $\operatorname{Fix}^c(f)$. This is an equivalence relation on $\operatorname{Fix}(f)$, defined as follows: Two fixed points $x$ and $y$ belong to the same class if and only if there exists a path $\alpha$ from $x$ to $y$ such that $\alpha \simeq f(\alpha)$ relative to the endpoints.

Before defining the Reidemeister trace, we first note that it lies in the free abelian group generated by the set of twisted conjugacy classes of the fundamental group, which is associated with the self-map $f$ on $M$. The \textit{twisted conjugacy classes} of the fundamental group form the quotient set $\pi_1(M,*)$ under the following equivalence relation:  
\[
 \beta \sim  \alpha \beta \phi(\alpha)^{-1}, \: \text{for all } \alpha, \beta \in \pi_1(M,*),
\]

where $\phi$ is the map induced by $f$. Explicitly, $\phi$ is defined by the composition of the following maps:  
\[
\phi\colon \pi_1(M,*) \xrightarrow{f_\#} \pi_1(M,f(*)) \xrightarrow{\cong} \pi_1(M,*).
\]

The isomorphism is established by choosing a homotopy class of paths $[\gamma]$ from the base-point $*$ to $f(*)$. Thus, $\phi$ sends $\alpha$ to $\gamma f_\#(\alpha) \gamma^{-1}$. It is important to note that the construction of $\phi$ is independent of the choice of $\gamma$. We denote the set of fundamental group elements modulo this twisted conjugacy relation by $\pi_1(M,*)_f$.

Now, for each fixed point $x$ of $f$, choose a path $\zeta_x$ from the base-point $*$ to $x$. This allows us to define a well-defined injective map from the set of fixed point classes to the twisted conjugacy classes of the fundamental group:  
\begin{align*}
\operatorname{Fix}^c(f) & \to \pi_1(M,*)_f\\
[x] & \mapsto \zeta_x * f(\zeta_x)^{-1} * \gamma^{-1}
\end{align*}

\begin{definition}
Let $f: M \to M$ be a map on a compact manifold $M$. The \textit{Reidemeister trace}, $R(f)$, of $f$ is defined as the image of the class  
\[
\sum_{[x] \in \operatorname{Fix}^c(f)} i(f,[x]) [x]
\]
under the injection  
\[
\mathbb{Z} [\operatorname{Fix}^c(f)]  \to \mathbb{Z} [\pi_1(M,*)_f],
\]
where $i(f,[x])$ denotes the fixed point index $i(f,U)$ of $f$ in an open set $U$ that contains all fixed points in the class $[x]$.
\end{definition}

\begin{definition}
The \textit{Nielsen number}, $N(f)$, of $f$ is defined by the number of non-zero coefficients of $R(f)\in\Z[\pi_1(M,*)_f]$.
\end{definition}

It is well known that for a given closed, compact, and connected manifold $M$, there is a correspondence between the connected components of $\Loop_f M$ and the twisted conjugacy classes of the fundamental group $\pi_1(M,*)_f$. Using this identification, we derive the following consequence of Theorem \ref{KWindex}. 

\begin{corollary}\label{ell(f)}
Let $f:M \to M$ be a map on closed, compact, smooth manifold. Then, the invariant $\ell(f)$ which is defined on \cite{KW1}, is identified with the Reidemeister trace $R(f)$.
\begin{align*}
\Omega_0^{\fr}(\Loop_f M) & \xrightarrow{\cong}  \Z[\pi_1(M,*)_f]\\
\ell(f) & \mapsto R(f)
\end{align*}
\end{corollary}

\begin{proof}
It is known that there exists an isomorphism of groups:
\begin{equation}\label{eq.proof.ell(f)}
\Omega_0^{\fr}(\Loop_f M) \xrightarrow{\cong} \mathbb{Z}[\pi_0(\Loop_f M)]
\end{equation}
The invariant $[\Gamma_{f} \pitchfork i_{\Delta}]$, which lies in $\Omega_0^{\fr}(\Loop_f M)$, is identified with $\ell(f)$ by Theorem \ref{KWindex}. Under this isomorphism, it corresponds to an element encoding the signed intersection numbers between the diagonal $\Delta_U$ of $U$ and the graph $\Gamma(f)$ of $f$, where $U \subseteq M$ is an open subset containing the fixed point class $[x]$. 

More precisely, under the isomorphism in \eqref{eq.proof.ell(f)}, $[\Gamma_{f} \pitchfork i_{\Delta}]$ maps to
\[
\sum_{[c_x] \in \pi_0(\Loop_f M)} ([\Delta_U] \cdot [\Gamma(f)]) [c_x],
\]
where $[\Delta_U] \cdot [\Gamma(f)]$ denotes the intersection number, and $[c_x]$ represents the constant path in $M$ based at $x \in \Fix(f)$.

Also, observe that two loops $c_x$ and $c_y$ in $\Loop_f M$ belong to the same component if and only if there exists a path $\alpha$ from $x$ to $y$ such that $\alpha \simeq f(\alpha)$. This condition precisely characterizes the fixed point classes.

Moreover, note that the signed intersection number between $\Delta_U$ and $\Gamma(f)$ corresponds to the fixed point index $i(f, U)$. Therefore, 
\begin{align*}
[\Gamma_{f} \pitchfork i_{\Delta}] \mapsto \sum_{[x] \in \operatorname{Fix}^c(f)} i(f,[x]) [c_x],
\end{align*}
via the isomorphism \eqref{eq.proof.ell(f)}. Furthermore, there is a one-to-one correspondence between following groups.
\begin{align*}
\Z[\pi_0(\Loop_f M)] & \xrightarrow{\cong} \Z[\pi_1(M,*)_f]\\
 [c_x] & \mapsto \zeta_x * f(\zeta_x)^{-1} * \gamma^{-1},
\end{align*}
where $\zeta_x$ is a chosen path from the based point $*$ to $x$. This gives the definition of the Reidemeister trace $R(f)$. Therefore, we obtained that
\begin{align*}
\Omega_0^{\fr}(\Loop_f M) & \to \Z[\pi_0(\Loop_f M)] \to \Z[\pi_1(M)_f]\\
\ell(f) & \mapsto \sum_{[x] \in \operatorname{Fix}^c(f)} i(f,[x]) [c_x] \mapsto R(f)
\end{align*}
\end{proof}

Next, we identify $\Omega_0^{\fr} (\Loop_f M^H)$ in terms of twisted conjugacy classes of the fundamental groups. For this purpose, we consider the connected components of $M^H$. Let $M^H_1, \ldots, M^H_k$ denote the connected components of $M^H$. Then, using the result above, we obtain the following identification:
\begin{align*}
\bigoplus_i \Omega_0^{\fr} (\Loop_f M^H_i)\cong \bigoplus_i \Z [\pi_0(\Loop_f M^H_i)] \xrightarrow{\cong} \bigoplus_i \Z [\pi_1(M^H_i,*)_{f^H_i}]
\numberthis\label{ident2}.
\end{align*}
By combining this with the tom Dieck splitting \eqref{tomDsplitting}, we derive a decomposition of $\Omega_0^{G,\fr}(\Loop_fM)$. With this groundwork in place, we will establish the connection between the equivariant invariant $\ell_G(f)$ and Reidemeister traces. To achieve this, we introduce an induced version of the Reidemeister trace.

\begin{definition}
Let $G$ be a finite group, which acts on a smooth, compact manifold $M$, and $f$ be a smooth self-map on $M$. Then, the \textit{reduced Reidemeister trace} $\overline{R(f)}$ is an element in $\Z[ \pi_1(M,*)_f]/G$ such that it is an image of $R(f)$ under quotient group homomorphism $\Z[ \pi_1(M,*)_f] \to \Z[ \pi_1(M,*)_f]/G$. That is, 
\begin{align*}
\overline{R(f)}= \sum_{\overline{[x]} \in \operatorname{Fix}^c(f)/G} \left( \sum_{[g x] \in G \cdot [x]}i(f,[g x]) \right) \overline{[x]}
\end{align*}
\end{definition}

First, we will show why this class is well-defined. The action $G$ on the set of fixed point classes is clearly given by $g[x]=[gx]$ for all $g \in G$, $[x] \in \operatorname{Fix}^c(f)$.

Now, let $[x]=[x']$ and $[y]=[y']$. If $[gx]=[y]$, then we need to show $[gx']=[y']$ for the well-definedness. Since $[x]=[x']$, there exists a path $\alpha$ from $x=f(x)$ to $x'=f(x')$ such that $\alpha \simeq f(\alpha)$. This implies that $[gx]=[gx']$ since $g(\alpha) \simeq g(f(\alpha))$. Similarly, there exists another path $\beta$ from $y$ to $y'$ such that $\beta \simeq f(\beta)$. If $[gx]=[y]$, then there is also a path $\gamma$ from $gx$ to $y$ such that $\gamma \simeq f(\gamma)$ Thus, the path $g(\alpha)^{-1} * \gamma * \beta$ from $gx'$ to $y'$ implies that $[gx']=[y']$ because $g(\alpha)^{-1} * \gamma * \beta$ and $f(g(\alpha)^{-1} * \gamma * \beta)$ are homotopic relative to the end points.

We can now state and prove the one of the main results of this paper.

\begin{theorem} \label{ident3}
The equivariant framed bordism group $\Omega_0^{G,\fr} (\Loop_f M)$ is isomorphic to the following group
\begin{align*}
    \bigoplus_{(H)} \left( \bigoplus_i \Z [\pi_1(M^H_i,*)_{f^H_i}] \right) \Big/ WH
\end{align*}
Furthermore, the invariant $\ell_G(f)$ maps to $\bigoplus_{(H), i} \overline{R(f^H_i)}$, where each $R(f_i^H)$ is the Reidemeister trace of the induced map $f^H_i: M^H_i \to M^H_i$, and $\overline{R(f^H_i)}$ is the reduced Reidemeister trace.
\end{theorem}

\begin{proof}
The first statement is clear from the isomorphism of the map, which is given above \eqref{ident2}. Combining Theorem \ref{splittingresult} with Corollary \ref{ell(f)}, we obtained the desired result: Each projection $\overline{\ell(f^H)}=\overline{\oplus_i \ell(f^H_i)}$ of $\ell_G(f)$ is identified with $\overline{\oplus_i R(f^H_i)}$.

\end{proof}

We have shown that $\ell_G(f)$ decomposes into reduced Reidemeister traces $\overline{R(f^H)}$. Since $R(f^H) = 0$ implies $\overline{R(f^H)} = 0$, it is clear that $R(f^H)$ contains at least as much information as $\overline{R(f^H)}$. This naturally raises the question of whether any information is lost when passing to the reduced version, particularly in the presence of the Weyl group action. It turns out that this is not the case: the reduced Reidemeister traces retain the same amount of information as the unreduced Reidemeister traces.

\begin{theorem}\label{quotientR(f)}
Given a finite group $G$ and a $G$-map $f: M \to M$ on a smooth $G$-manifold $M$, we have the following equivalence between the (unreduced) Reidemeister trace and the reduced one:
$$
R(f) = 0 \quad \text{if and only if} \quad \overline{R(f)} = 0.
$$
More precisely, the reduced Reidemeister trace is given by
$$
\overline{R(f)} = \sum\limits_{\overline{[x]} \in \operatorname{Fix}^c(f)/G} |G \cdot [x]| \cdot i(f,[x]) \cdot \overline{[x]} \in \mathbb{Z}[\pi_1(M, *)_f]/G.
$$
Consequently, the Nielsen number satisfies that
$$
N(f) = 0 \quad \text{if and only if} \quad \overline{N(f)} = 0.
$$
\end{theorem}

\begin{proof}
The "if" statement is clear. The converse is the direct consequence of the Proposition \ref{fix.point.index} given by below.
\end{proof}

\begin{proposition}\label{fix.point.index}
Given an equivariant map $f: X \to X$, where $X$ is a $G$-space, the fixed point index $i(f,[x])$ of the point $x \in X$ is the same as the fixed point index $i(f,[g x])$ for all $g \in G$.
\end{proposition}

\begin{proof}
For simplicity, we denote the fixed point indices as follows: $i(x):=i(f,[x])$ and $i(g x):=i(f,[g x])$. Recall that the fixed point index is defined locally. Let $U$ be an open set in $X$ such that $F \subseteq U$, where $F$ is the fixed point class $[x]$ of $x$. Denote the chart map as $\psi\colon U \to \R^n$, and let $V$ be an open $n$-ball neighborhood of $F$ in $U$ such that $f(V) \subseteq U$. Then, the fixed point index $i(x)$ is defined by the following induced map:
\begin{align*}
    (\mathrm{id}- \psi \circ f \circ \psi^{-1})_* \colon H_n(\psi(V), \psi(V)-\psi(F)) \to H_n(\R^n,\R^n-\{0\}) \cong \Z
\end{align*}
More precisely, 
$$
(\mathrm{id}- \psi \circ f \circ \psi^{-1})_*([S^{n}]_F)=i(x)\cdot 1.
$$

Now, we can define $i(g x)$ by using the maps that are given to define $i(x)$. Consider the set $gU=\{g x\mid x \in U\}$. Clearly, this is an open set in $X$ since each $g \in G$ is a homeomorphism. Also, it is easy to see that $gF=\{g  x\mid x\in F\} \subseteq gU$ coincides with the fixed point class of $g x$. We can take the chart map as $\phi\colon gU \xrightarrow{g^{-1}} U \xrightarrow{\psi} \R^n$. Moreover, $gV$ is an open neighborhood of $gF$ in $gU$ such that $f(gV) \subseteq gU$ because $gf(V) \subseteq gU$. Thus, $i(g x)$ can be defined by the induced map below:
\begin{align*}
    (\mathrm{id}- \psi \circ g^{-1} \circ f \circ g \circ \psi^{-1})_* \colon H_n(\phi(gV), \phi(gV)-\phi(gF)) \to H_n(\R^n,\R^n-\{0\}) \cong \Z
\end{align*}

Since $g^{-1} \circ f \circ g = f$ and $(\phi(gV), \phi(gV)-\phi(gF))=(\psi(V), \psi(V)-\psi(F))$, the fixed point indices $i(x)$ and $i(g x)$ are equal.
\end{proof}

\section{An Application to Periodic Points}\label{sec5}

Let $M$ be a closed manifold, and let $f \colon M \to M$ be a self-map. Given an integer $n \geq 2$, a point $x \in M$ is called \textit{$n$-periodic} if it is a fixed point of the $n$th iterate of $f$, that is, $f^n(x) = x$.  

The \textit{homotopical periodic point problem} asks whether a given self-map can be deformed into another map that has no $n$-periodic points. An obstruction theory for this problem was developed in \cite[Theorem J]{KW2}, using the equivariant fixed point problem for the \textit{Fuller map} $\Phi_n(f)$ of $f$:
\[
\begin{aligned}
    \Phi_n(f): M \times \cdots \times M &\to M \times \cdots \times M \\
    (x_1, \ldots, x_n) &\mapsto (f(x_n), f(x_1), \ldots, f(x_{n-1})).
\end{aligned}
\]
There is a natural cyclic group action of $\mathbb{Z}_n = \langle g \rangle$ on $M^n$, given by  
\[
g \cdot (x_1, \ldots, x_n) = (x_n, x_1, \ldots, x_{n-1}).
\]
It is straightforward to verify that the Fuller map $\Phi_n(f)$ is $\mathbb{Z}_n$-equivariant. Furthermore, there exists a bijective $\mathbb{Z}_n$-equivariant correspondence between the $n$-periodic points of $f$ and the fixed points of $\Phi_n(f)$, given by  
\[
x \mapsto (x, f(x), \dots, f^{n-1}(x)).
\]
In particular, the Fuller map $\Phi_n(f)$ is fixed-point free if and only if $f$ has no $n$-periodic points.  

We define the \textit{homotopy $n$-periodic point set} $\ho P_n(f)$ of $f$ as the set of $n$-tuples $(\lambda_1, \lambda_2, \ldots, \lambda_n)$ of paths $\lambda_i$ in $M$ satisfying the condition  
\[
f(\lambda_i(0)) = \lambda_{i+1}(1),
\]
for all $i$ modulo $n$. The action of $\mathbb{Z}_n$ on $\ho P_n(f)$ is given by cyclic permutation of components:  
\[
g \cdot (\lambda_1, \lambda_2, \ldots, \lambda_n) = (\lambda_n, \lambda_1, \ldots, \lambda_{n-1}).
\]

We now state the result of Klein and Williams concerning the $n$-periodic point problem.

\begin{theorem}{\cite{KW2}} \label{thmI}
There is a homotopy theoretically defined invariant
\begin{align*}
    \ell_n(f) \in \Omega_0^{\Z_n,fr}(\ho P_n(f))
\end{align*}
which is an obstruction to deforming $f$ to an n-periodic point free self map.
\end{theorem}

This invariant $\ell_n(f)$ does not serve as a complete obstruction, as the authors were unsure whether the statement ``the Fuller map $\Phi_n(f)$ is fixed-point free if and only if $f$ is $n$-periodic point free" holds up to homotopy. It is evident that if $f$ can be deformed to $f'$, where $f'$ has no $n$-periodic points, then $\Phi_n(f)$ is homotopic to $\Phi_n(f')$, which is fixed-point free. On the other hand, there is no homotopical justification for why the reverse implication should necessarily hold true. Indeed, if the diagram \eqref{0cartesian}, described below, is $0$-cartesian, then the $n$-periodic point problem for $f$ reduces to the $\mathbb{Z}_n$-equivariant fixed-point problem for $\Phi_n(f)$.

Let $\operatorname{end}(M)$ be the space of self-maps of $M$, and let $\operatorname{end}^{\star_n}(M) \subseteq \operatorname{end}(M)$ denote the subspace consisting of self-maps that have no $n$-periodic points. Similarly, we denote the space of $\mathbb{Z}_n$-equivariant self-maps of $M^n$ by $\operatorname{end}(M^n)^{\mathbb{Z}_n}$, and let $\operatorname{end}^\star(M^n)^{\mathbb{Z}_n} \subseteq \operatorname{end}(M^n)^{\mathbb{Z}_n}$ be the subspace of equivariant self-maps of $M^n$ that have no fixed points.  

There is an operator, called \textit{Fuller transform} $\Phi_n$, defined by
\begin{align*}
    \Phi_n \colon \operatorname{end}(M) & \to \operatorname{end}(M^n)^{\Z_n}\\
    f & \mapsto \Phi_n (f).
\end{align*}

The commutative diagram is given as follows. The vertical maps are the Fuller transforms and the horizontal ones are inclusions.

\begin{equation}\label{0cartesian}
\begin{tikzcd}
    \& \operatorname{end}^{\star_n}(M) \arrow[r,""] \arrow[d,"\Phi_n"] \& \operatorname{end}(M)  \arrow[d,"\Phi_n"]\\
    \& \operatorname{end}^\star(M^n)^{\Z_n} \arrow[r,""] \& \operatorname{end}(M^n)^{\Z_n}
\end{tikzcd}    
\end{equation}

This square is indeed cartesian. By the universal property of a homotopy pullback, there is a natural map from $\operatorname{end}^{\star_n}(M)$ to the homotopy pullback of the diagram  
\[
\operatorname{end}^\star(M^n)^{\Z_n} \xrightarrow{ } \operatorname{end}(M^n)^{\Z_n} \xleftarrow{\Phi_n} \operatorname{end}(M).
\]

We denote this map as follows.
\[
\psi\colon \mathrm{end}^{\star_n}(M) \to  \operatorname{end}^\star(M^n)^{\Z_n}\times_{\operatorname{end}^\star(M^n)^{\Z_n}}^h \mathrm{end}(M).
\]

If the square \eqref{0cartesian} is $0$-cartesian, then the induced map of $\psi$ on connected components is surjective, meaning
\[
\psi\colon \pi_0(\operatorname{end}^{\star_n}(M)) \twoheadrightarrow \pi_0( \operatorname{end}^\star(M^n)^{\Z_n}\times_{\operatorname{end}^\star(M^n)^{\Z_n}}^h \operatorname{end}(M)).
\]

The universal property of the pullback defines the map $\psi$.
\begin{equation*}
\begin{tikzcd}
\operatorname{end}^{\star_n}(M) \arrow[ddr,bend right,"\Phi_n"'] \arrow[drr,bend left,"i"] \arrow[dr,dashed,"\psi"] \\
\&   \operatorname{end}^\star(M^n)^{\Z_n} \times_{\operatorname{end}^\star(M^n)^{\Z_n}}^h \operatorname{end}(M)\arrow[r,"\mathrm{pr}_1"] \arrow[d,"\mathrm{pr}_2"] \& \operatorname{end}(M) \arrow[d,"\Phi_n"] \\
\& \operatorname{end}^\star(M^n)^{\Z_n} \arrow[r,"i"] \& \operatorname{end}(M^n)^{\Z_n}
\end{tikzcd}
\end{equation*}

From the diagram above, we have that  
$$\psi: f \mapsto (\Phi_n(f), c_{\Phi_n(f)},f),$$
where $c_{\Phi_n(f)}$ is the constant path on the mapping space $\operatorname{end}^\star(M^n)^{\Z_n}$. Now, assume that the induced map $\psi$ on the connected component is surjective map. Then, for any $[(\theta, \lambda, g)] \in \pi_0( \operatorname{end}^\star(M^n)^{\Z_n}\times_{\operatorname{end}^\star(M^n)^{\Z_n}}^h \operatorname{end}(M))$, there exists $f \in \operatorname{end}^{\star_n}(M)$ such that
\[[(\theta, \lambda, g)]=[(\Phi_n(f),c_{\Phi_n(f)},f)]\]

From here, one can conclude that if $\Phi_n(g) \simeq \Phi_n(f)$ for $\Phi_n(g) \in \operatorname{end}(M^n)^{\Z_n}$ and $\Phi_n(f) \in \operatorname{end}^\star(M^n)^{\Z_n}$, then $g \simeq f$  for $g\in \operatorname{end}(M)$ and $f \in \operatorname{end}^{\star_n}(M)$. As a result, $\Phi_n(f)$ is fixed-point free if and only if $f$ is $n$-periodic point free up to homotopy provided that the diagram \eqref{0cartesian} is $0$-cartesian. However, as noted earlier, we are uncertain whether it satisfies the $0$-cartesian property. Nevertheless, one can still define an invariant of self maps which is trivial when the self map is homotopic to an n-periodic point free one. This invariant is the one $\ell_n(f)$, given in Theorem \ref{thmI}, and it is defined as follows.
\[
\ell_n(f):=\ell_{\Z_n} (\Phi_n(f))
\]

Note that $\ell_n(f)$ lives in $\Omega_0^{\Z_n,fr}(\Loop_{\Phi_n(f)}M^n)$, which is isomorphic to $\Omega_0^{\Z_n,fr}(\ho P_n)$ since
\[
\Loop_{\Phi_n(f)}(M^n) \cong \ho P_n.
\]

It is now clear that if $f$ is homotopic to an $n$-periodic point free map, then $\ell_n(f)$ vanishes. This is the direct consequence of Theorem \ref{KWfix} as we know that if $f$ is homotopic to an $n$-periodic point free map, then $\Phi_n(f)$ is homotopic to fixed-point free map equivariantly. This proves Theorem \ref{thmI}.

Furthermore, if we apply the converse of Theorem \ref{KWfix} to $\Phi_n(f)$, then we obtain the following conclusion.

\begin{corollary}{\cite[11.2]{KW2}}\label{cor.thmH}
If $\dim M \geq 3$ and $\ell_n(f)=0$, then $\Phi_n(f)$ is equivariantly homotopic to a fixed point free map.
\end{corollary}

\begin{proof}
It is enough to show that when $\dim M \geq 3$, we have  $\dim(M^n)^{\Z_k} \geq 3$ and 
$$\dim(M^n)^{\Z_l}-\dim(M^n)^{\Z_k}\geq 2,$$
for all $k|n$ and $l|k$ such that $l \neq 1, k$. 

Note that $(M^n)^{\Z_k} \cong M^{\frac{n}{k}}$. Therefore, $\dim((M^n)^{\Z_k}) \geq 3$. Moreover, since $l|k$ and $l\neq k$, we obtain that 
\[
\dim(M^n)^{\Z_l}-\dim(M^n)^{\Z_k}= \dim M^{\frac{n}{l}} -\dim M^{\frac{n}{k}}\geq \dim M >2.
\]
\end{proof}

There is another solution for the $n$-periodic point problem, which is developed by Jezierski \cite{jezierski}.

\begin{theorem}{\cite[2.3]{jezierski}}\label{Jezierski}
Let $f : M \rightarrow M$ be a self map of a closed, smooth manifold $M$ such that $\dim M\geq3$. Then, $f$ is homotopic to a map $g$ without $n$-periodic points if and only if $N(f^k)=0$ for any divisor $k$ of $n$, where $f^k$ is the $k$-fold composition of $f$ with itself.
\end{theorem} 

Now, we pose the following question: Does the Klein-Williams invariant $\ell_n(f)$ contain at least as much information as the Nielsen numbers $N(f^k)$?  

Klein and Williams conjectured that this is indeed the case. Their result, Theorem \ref{thmJ}, suggests that each projection of $\ell_n(f)$ under the tom Dieck splitting lies in an abelian group that contains the Reidemeister trace $R(f^k)$.  

\begin{theorem}{\cite{KW2}} \label{thmJ}
The abelian groups given as follows are isomorphic.
   \begin{align*}
   \Omega_0^{\Z_n,fr}(\ho P_n(f)) \cong \bigoplus_{k|n} \Z[\pi_1(M)_{\rho,k}]
   \end{align*}
where $\pi_1(M)_{\rho, k}$ is the set of equivalence classes on $\pi_1(M,*)$ generated by the following relations
\[\alpha \sim \beta \alpha \rho(\alpha)^{-1} \text{   and  }  \alpha \sim \rho(\alpha) \text{   for all  } \alpha,\beta \in \pi_1(M,*),\]
the map $\rho: \pi_1(M,*) \rightarrow \pi_1(M,*)$ is defined by $\alpha \mapsto w f(\alpha)w^{-1}$ for a chosen path $w$ from $*$ to $f(*)$.
\end{theorem}

\begin{conjecture}{\cite[11.3]{KW2}}\label{conj}
Let $\mathcal{N}_k(f)$ be the number of non-zero terms in $\ell_n(f)$
expressed as a linear combination of the basis elements of $\Z[\pi_1(M)_{\rho,k}]$. Then, $\mathcal{N}_k(f)=N(f^k)$.
\end{conjecture}

Our goal is to establish the relationship between $\mathcal{N}_l(f)$ and $N(f^l)$ by providing an explicit description of $\overline{R(\Phi_n(f)^{\Z_k})}$. Although $\mathcal{N}_l(f)$ and $N(f^l)$ do not always coincide (see Example \ref{ex.periodic}), we conclude that $\mathcal{N}_l(f) = 0$ if and only if $N(f^l) = 0$ in Corollary \ref{proof.conj}.

Recall Theorem \ref{ident3}, where we obtained a decomposition of the Klein-Williams invariant $\ell_G(f)$ for the equivariant fixed-point problem under the tom Dieck splitting. From this result, we deduce that

\[
\ell_n(f) = \ell_{\Z_n}(\Phi_n(f)) = \bigoplus_{k|n} \overline{R(\Phi_n(f)^{\Z_k})} \in \bigoplus_{k|n} \Z[\pi_1((M^n)^{\Z_k},*)_{\Phi_n(f)}]/W\Z_k.
\]

Observe that each component $\Z[\pi_1((M^n)^{\Z_k},*)_{\Phi_n(f)}]/W\Z_k$ contributes in the same way as in Theorem \ref{thmJ}.

\begin{lemma}\label{lemmaJ} Let $k|n$, and $kl=n$. Then, there is an isomorphism of abelian groups.
    \[\Z[\pi_1((M^n)^{\Z_k},*)_{\Phi_n(f)}]/W\Z_k \cong \Z[\pi_1(M)_{\rho,l}]\]
\end{lemma}

Note that 
\[
(M^n)^{\Z_k}=\{(x_1,\ldots,x_l,x_1,\ldots,x_l,\ldots,x_1,\ldots, x_l)\in M^n \mid l=\frac{n}{k}\}
\]

For simplicity, denote $\Phi_n(f)^{\Z_k}$ by $\Phi^k_n(f)$, which is defined as
\begin{align*}
    \Phi^k_n(f)\colon (M^n)^{\Z_k} & \to (M^n)^{\Z_k}\\
    (x_1,\ldots,x_l,\ldots,x_1,\ldots, x_l) & \mapsto (f(x_l), f(x_1),\ldots,f(x_{l-1}), f(x_1)\ldots,f(x_{l-1})).
\end{align*}

\begin{proof}[Proof of Lemma \ref{lemmaJ}]
It is clear that $(M^n)^{\Z_k} \cong M^l$. Therefore, we identify $(M^n)^{\Z_k}$ with $M^l$ for simplicity. Under this identification, the map $\Phi_n^k(f)$ coincides with the map $\Phi_l(f) \colon M^l \to M^l$. Recall that the set $\pi_1(M^l,*)_{\Phi_n(f)}$ is given by the twisted conjugacy relations, defined as follows.
\[\
\alpha^{-1} \beta w \Phi_l(f)(\alpha)w^{-1}\sim \beta \text{  for all  } \alpha,\beta \in \pi_1(M^l,*),
\]
where $w$ is the chosen path from the based point $*$ to $\Phi_l(f)(*)$, which defines an isomorphism $\pi_1(M^l,*)\cong\pi_1(M^l,\Phi_l(f)(*))$. One can choose the based point $*$ as $(*,\ldots,*)$ without loss of generality, and let $w=(w,\ldots,w)$, where each $w$ is a path on $M$ from $*$ to $f(*)$. Then, the set of equivalence classes on $\pi_1(M^l,*)$ becomes
\[
\pi_1(M,*) \times \pi_1(M,*)\times \cdots\times \pi_1(M,*)
\]
with respect to the equivalence relation
\[
(\beta_1,\ldots,\beta_l)\sim(\alpha_1^{-1} \beta_1 w f(\alpha_l)w^{-1},\alpha_2^{-1} \beta_2 w f(\alpha_1)w^{-1},\ldots,\alpha_l^{-1} \beta_l w f(\alpha_{l-1})w^{-1})
\]
for all $\alpha_i,\beta_i\in \pi_1(M,*)$. This set of equivalence classes gives rise to $\Z[\pi_1(M)_{\rho,l}]$ due to the same arguments in the Proof of Theorem \ref{thmJ}.
\end{proof}

To establish the relationship between $\overline{R(\Phi_n^k(f))}$ and the Nielsen numbers $N(f^k)$, we begin by examining the fixed point set of $\Phi_n^k(f)$. It is straightforward to verify that 
\[
\Fix (\Phi^k_n(f))=\{(x,f(x),\ldots,f^{l-1}(x),x,\ldots,f^{l-1}(x))\in M^n \mid f^l(x)=x\}
\]

The following lemma will be useful in comparing the Nielsen numbers $N(\Phi_n^k(f))$ and $N(f^l)$.

\begin{lemma}\label{correspond.of.fix.pt.cl}
    There exists a one-to-one correspondence between the set of fixed point classes of $\Phi_n^k(f)$ and the set of fixed point classes of $f^l$, where $l=\frac{n}{k}$.
    \begin{align*}
        \operatorname{Fix}^c(\Phi_n^k(f))& \to \operatorname{Fix}^c(f^l)\\
        [(x,f(x),\ldots,f^{l-1}(x))] & \mapsto [x]
    \end{align*}
\end{lemma}

\begin{proof}
Assume that 
$$[(x,f(x),\ldots,f^{l-1}(x))]=[(y,f(y),\ldots,f^{l-1}(y))].$$
Then, there exits a path
$$\vec{\alpha}=(\alpha_1,\ldots,\alpha_l)$$
in $(M^n)^{\Z_k}$ from $(x,f(x),\ldots,f^{l-1}(x))$ to $(y,f(y),\ldots,f^{l-1}(y))$ such that
\[
\Phi_n^k(f)(\alpha)=(f(\alpha_l),f(\alpha_1),\ldots,f(\alpha_{l-1}))\simeq(\alpha_1,\alpha_2,\ldots,\alpha_l).
\]
Therefore, $\alpha_2\simeq f(\alpha_1)$ implies that $\alpha_3\simeq f(\alpha_2)\simeq f^2(\alpha_1)$. Inductively, we have $\alpha_l\simeq f(\alpha_{l-1})\simeq f^{l-1}(\alpha_1)$, and so $\alpha_1 \simeq f^l(\alpha_1)$, where $\alpha_1$ is a path from $x$ to $y$. As a result, $[x]=[y]$; thus, this map is well-defined.

Now, assume that $[x]=[y]$, then there exits a path $\alpha\colon\text{I}\to M$ from $x$ to $f(x)$ such that $f^l(\alpha)=\alpha$. Then, define a path as follows.
$$\vec{\alpha}:=(\alpha,f(\alpha),\ldots,f^{l-1}(\alpha))\colon \text{I}\to(M^n)^{\Z_k}.$$
This path is from $(x,f(x),\ldots,f^{l-1}(x))$ to $(y,f(y),\ldots,f^{l-1}(y))$. Also, we have
$$\Phi_n^k(f)(\alpha)=(f^l(\alpha),f(\alpha),\ldots,f^{l-1}(\alpha))\simeq(\alpha,f(\alpha),\ldots,f^{l-1}(\alpha))$$
since $f^l(\alpha)\simeq\alpha$. Therefore,
$$[(x,f(x),\ldots,f^{l-1}(x))]=[(y,f(y),\ldots,f^{l-1}(y))].$$ 
Hence, the map is injective. Lastly, this map is surjective because if $[x]$ is a fixed point class of $f^l$, then $[(x,f(x),\ldots,f^{l-1}(x))]$ is a fixed point class of $\Phi_n^k(f)$.
\end{proof}

From Lemma \ref{correspond.of.fix.pt.cl}, one can conclude that geometric Reidemeister trace of $\Phi_n^k(f)$ can be given as follows.
\begin{align*}
R(\Phi_n^k(f)) = & \sum_{[x]\in \text{Fix}^c(f^l)} i(\Phi_n^k(f), [(x,f(x),\ldots,f^{l-1}(x))])[(x,f(x),\ldots,f^{l-1}(x))]\\
&\in \Z[\pi_1((M^n)^{\Z_k},*)_{\Phi_n^k(f)}]
\end{align*}

The next proposition shows the relationship between the fixed point index of $\Phi_n^k(f)$ at the point $(x,f(x),\ldots,f^{l-1}(x))$ and the fixed point index of $f^l$ at $x$.

\begin{proposition}\label{correspond.index} Let $f^l$ have only generic fixed points; that is, $\det (\mathrm{I}-D_xf^l)\neq 0$ for all $x\in \operatorname{Fix}(f^l)$. Then,
\[i(\Phi_n^k(f), [(x,f(x),\ldots,f^{l-1}(x))])=i(f^l,[x])\]
\end{proposition}

\begin{proof}
Since $(M^n)^{\mathbb{Z}_k} \cong M^l$, and $\Phi_n^k(f)$ corresponds to $\Phi_l(f)$ on $M^l$, it suffices to show that  
\[
i(\Phi_l(f), (x,f(x),\ldots,f^{l-1}(x))) = i(f^l,x).
\]  

Since $f^l$ has generic fixed points, the fixed point index can be calculated as follows (see \cite{jiangbook}, Chapter \RN{1}, 3.2 for the details).
\[
i(f^l,x) = \operatorname{sign} \det(\text{I} - D_x f^l) = (-1)^t,
\]  
where $t$ is the number of real eigenvalues of $D_x f^l$ greater than 1. Furthermore, we can conclude that $\Phi_l(f)$ also has generic fixed points, just as $f^l$ does. Indeed, we aim to show that  
\[
\det (\text{I} - D_x f^l) = \det(\text{I} - D_{\vec{x}} \Phi_l(f)).
\]  

To this end, we first consider the Jacobian matrix of $\Phi_l(f)$ at the point $\vec{x} = (x, f(x), \ldots, f^{l-1}(x))$, which is given by the following matrix. 

\[
\begin{pmatrix}
    0 & 0 & \cdots & 0 & D_{f^{l-1}(x)}f\\
    D_xf & 0 & \cdots & & 0\\
    0 & D_{f(x)}f & 0 & \cdots &0\\
    \vdots & \ddots & \ddots & \ddots & \vdots\\
    0 & \cdots &0 & D_{f^{l-2}(x)}f & 0
\end{pmatrix}_{ml \times ml}
\]
Note that each entry is an $m \times m$ block matrix, where $m = \dim M$. We will compute $\det (\text{I} - D_{\vec{x}} \Phi_l(f))$ using a property of the determinant of block matrices. Let  
\[
\det (\text{I} - D_{\vec{x}} \Phi_l(f)) =
\begin{vmatrix}
\text{I} & V \\
W & D
\end{vmatrix},
\]  
where $\text{I}$ is the $m \times m$ identity matrix, and the blocks are given by  

\[
V =
\begin{bmatrix}
0 & \cdots & 0 & -D_{f^{l-1}(x)} f
\end{bmatrix},
\]  

\[
W =
\begin{bmatrix}
-D_x f \\
0 \\
\vdots \\
0  
\end{bmatrix},
\]  

\[
D =
\begin{bmatrix}
\text{I} & 0 & \cdots & 0\\
-D_{f(x)} f & \text{I} & \cdots & 0\\
\vdots & \ddots & \ddots & \vdots\\
0 & 0 & -D_{f^{l-2}(x)} f & \text{I}
\end{bmatrix}.
\] 
Then, using the determinant formula for block matrices, we obtain  
\[
\det (\text{I} - D_{\vec{x}} \Phi_l(f)) = \det(\text{I}) \det(D - W \text{I}^{-1} V),
\]  
which simplifies to  

\[
\begin{vmatrix}
    \text{I} & 0 & \cdots & -D_x f D_{f^{l-1}(x)} f\\
    -D_{f(x)} f & \text{I} & \cdots & 0\\
    \vdots & \ddots & \ddots & \vdots\\
    0 & 0 & -D_{f^{l-2}(x)} f & \text{I}
\end{vmatrix}.
\]  
Using the same method inductively, we obtain  

\[
\det (\text{I} - D_{\vec{x}} \Phi_l(f)) =
\begin{vmatrix}
\text{I} & D_{f^{l-3}(x)} f \cdots D_x f D_{f^{l-1}(x)} f \\
- D_{f^{l-2}(x)} f & \text{I}
\end{vmatrix}.
\]  
Taking the determinant, we obtain that

\[
\det(\text{I} - D_{f^{l-2}(x)} f D_{f^{l-3}(x)} f \cdots D_x f D_{f^{l-1}(x)} f).
\]  
Now, observe that  

\[
\det (\text{I} - D_x f^l) = \det (\text{I} - D_x f D_{f(x)} f \cdots D_{f^{l-1}(x)} f),
\]  
which is exactly the same expression as above. This shows that  

\[
i(\Phi_l(f), (x, f(x), \ldots, f^{l-1}(x))) = i(f^l, x).
\]  
\end{proof}

A direct conclusion of this Proposition \ref{correspond.index} and Lemma \ref{correspond.of.fix.pt.cl} gives the result below.

\begin{corollary}\label{Reidemeister.traces.same}
The Nielsen numbers of $\Phi_n^k(f)$ and $f^l$ are the same when they have generic fixed points. In particular, let
\begin{align*}
R(f^l)= \sum_{[x]\in \operatorname{Fix}^c(f^l)} i(f^l,[x])[x] \in   \Z[\pi_1(M,x)_{f^l}].
\end{align*}
Then,
\begin{align*}
R(\Phi_n^k(f))= \sum_{[x]\in \operatorname{Fix}^c(f^l)} i(f^l,[x])[(x,f(x),\ldots,f^{l-1}(x))] \in \Z[\pi_1((M^n)^{\Z_k},\vec{x})_{\Phi_n^k(f)}].
\end{align*}
\end{corollary}

From the above result, we have concluded an identification between Reidemeister traces $R(\Phi_n^k(f))$ and $R(f^l)$. On the other hand, we need to take into account the action of the Weyl group $W\Z_k$ since the projection of $\ell_n(f)$ on each component is $\overline{R(\Phi_n^k(f))}$ under the tom Dieck splitting. 

\begin{remark}
If $x$ and $f^i(x)$ are in the same fixed point class of $f^l$ for all $n$-periodic points $x$ and for all $1\leq i <l$, then the Reidemeister trace $R(\Phi_n^k(f))$ under the quotient map 
\[
\Z [\pi_1((M^n)^{\Z_k},*)_{\Phi_n^k(f)}] \to \Z [\pi_1((M^n)^{\Z_k},*)_{\Phi_n^k(f)}/W\Z_k]
\] 
does not change. In other words, $\overline{R(\Phi_n^k(f))}=R(\Phi_n^k(f))$.
\end{remark}
    
On the other hand, this is not true when $x$ and $f^i(x)$ are not in the same fixed-point classes for some $i \pmod l$. Example \ref{ex.periodic} shows that $\overline{R(\Phi_n^k(f))} \neq R(\Phi_n^k(f))$, and thus, $\mathcal{N}_l(f) \neq N(f^l)$, contrary to the suggestion in Conjecture \ref{conj}. Nevertheless, it is still true that the two invariants contain the same information, as indicated by the following Corollary, which implies that $\ell_n(f)$ vanishes if and only if the Nielsen numbers $N(f^k)$ vanish for all $k \mid n$.

\begin{corollary}\label{proof.conj}
$\mathcal{N}_l(f)$ vanishes if and only if $N(f^l)$ does.
\end{corollary}

\begin{proof}
$\mathcal{N}_l(f)$ vanishes if and only if $\overline{R(\Phi_n^k(f))}$ vanishes. Using Theorem \ref{quotientR(f)}, one can deduce that $\overline{R(\Phi_n^k(f))}$ vanishes if and only $R(\Phi_n^k(f))$, and $R(\Phi_n^k(f))$ vanishes if and only if $N(f^l)$ vanishes by Proposition \ref{correspond.index}.
\end{proof}

Corollary \ref{proof.conj} was previously established in \cite{MalkiewichPonto} through the use of duality and trace in the context of bicategories. Their approach also enables direct generalizations, such as an analogous result for the fiberwise Reidemeister trace. In contrast, our proof here employs more classical techniques. Furthermore, we construct an explicit example demonstrating that $\mathcal{N}_l(f) \neq N(f^l)$.

\begin{example}\label{ex.periodic}
Consider the manifold $Y \subseteq S^2=\{(x,y,z)\mid x^2+y^2+z^2=1\}\subseteq \R^3$, which is a sphere with four open balls removed, as shown in Figure \ref{sphere-4balls} below.

\begin{figure}[h]
    \centering
    \includegraphics[width=0.5\linewidth]{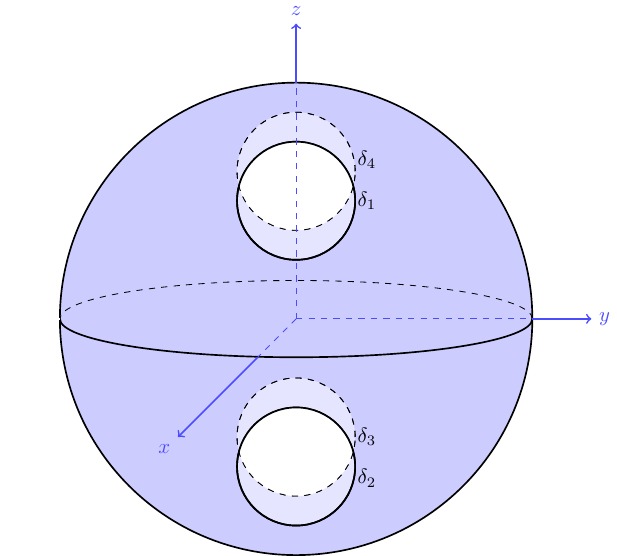}
    \caption{The manifold $Y$}\label{sphere-4balls}
\end{figure}

First, we define a $\Z_4 \times \Z_2$-action on $S^2$ as follows. The generator $g$ of $\Z_4$ rotates each point by $\frac{\pi}{2}$-radians about the $y$-axis, and the generator $h$ of $\Z_2$ reflects each point across the $xz$-plane.
This group acts on the boundary components $\delta_1, \delta_2, \delta_3, \delta_4$ such that $h(\delta_i) = \delta_i$ and $g(\delta_i) = \delta_{i+1}$  for all integers $i$ modulo $4$.
Since $k \cdot y \in Y$ for all $k \in \Z_4 \times \Z_2$ and $y \in Y$, this induces a $\Z_4 \times \Z_2$-action on $Y$.

Consider the map $g\circ h$ on $Y$, which has two $2$-periodic points. Since we aim to work with compact manifolds, we take two copies of $Y$ and glue them along their boundaries. Denote this compact manifold as $X$. Clearly, $g\circ h \cup g\circ h$ defines an endomorphism on $Y \cup_{\partial Y} Y = X$. We call this function $f$.
\[
f:= g\circ h \cup g\circ h \colon Y\cup_{\partial Y} Y \to Y\cup_{\partial Y} Y 
\]

Now, we will first compute $\ell_2(f)$ and then determine the Nielsen numbers $N(f)$ and $N(f^2)$ to compare them with the coefficients of $\ell_2(f)$.

Since $R(f^2)=R(\Phi_2(f))$, by Corollary \ref{Reidemeister.traces.same}, it is enough to compute $R(f^2)$. From the additivity property (also known as the pushout formula) of the Reidemeister traces (see \cite{Ferrario1999} for more details), we have
\[
R(f^2)=R(f^2_{|Y})+R(f^2_{|Y})-R(f^2_{|\partial Y}).
\]

Note that there is no $2$-periodic points on $\partial Y$; therefore, we only consider $R(f^2_{|Y})$. There are $2$-generic fixed points of $f^2_{|Y}$, denote one of them as $x$, then the other one is $f(x)$. It is easy to check that $x$ and $f(x)$ are not the in same fixed point class because there is no such a path $\alpha$ satisfies $\alpha \simeq f^2(\alpha)$ on $Y$.

Now, to calculate $i(f^2,x)$, take an closed ball $V\subseteq Y$, that contains the point $x$. Then, the induced map of $f^2$ on $\partial V \cong S^2 \subseteq \R^3$ is given by antipodal map on $S^2$; and therefore,
\[
i(f^2,x)=\frac{\vec{y}-f^2(\vec{y})}{||\vec{y}-f^2(\vec{y})||}=\frac{\vec{y}-(-\vec{y})}{||\vec{y}-(-\vec{y})||}=1.
\]
One can show that $i(f^2, x) = i(f^2, f(x))$. In fact, a more general statement holds:  
\[
i(f^l, x) = i(f^l, f^i(x)) \quad \text{for all } 1 \leq i \leq l.
\]
To show that $i(f^l, x) = i(f^l, f(x))$, we compare the Jacobian matrices $D_x f^l$ and $D_{f(x)} f^l$. By the chain rule, we have  
\[
D_x f^l = D_{f^{l-1}(x)} f \cdot D_{f^{l-2}(x)} f \cdots D_{f(x)} f \cdot D_x f.
\]
Similarly,  
\[
D_{f(x)} f^l = D_x f \cdot D_{f^{l-1}(x)} f \cdots D_{f^2(x)} f \cdot D_{f(x)} f.
\]
Since these matrices are identical, it follows that  
\[
i(f^l, x) = i(f^l, f(x)).
\]
Thus, using the fact that $i(f^2, x) = i(f^2, f(x))$, we proceed to compute the Reidemeister trace.
\[
R(f^2)=R(f^2_{|Y})+R(f^2_{|Y})=[x]+[f(x)]+[x']+[f(x')]
\]
This implies that
\[
R(\Phi_2(f))=[(x,f(x))]+[(f(x),x)]+[(x',f(x'))]+[(f(x'),x')]
\]
Since $\Z_2=\langle k\rangle$ acts on $X\times X$ as $k\cdot (x,f(x))=(f(x),x)$, we obtain that
\[
\overline{R(\Phi_2(f))}=\overline{[(x,f(x))]}+\overline{[(x',f(x'))]}.
\]
Now, we will verify $R(\Phi_2(f)^{\Z_2})$. Note that $\Phi_2(f)^{\Z_2} \colon \Delta_X \to \Delta_X$ is given by $$(x,x)\mapsto(f(x),f(x)).$$
Since there is no fixed point of $f$, it is clear that $R(\Phi_2(f)^{\Z_2} )=0$. Thus,
\[
\ell_2(f)=\overline{[(x,f(x))]}+\overline{[(x',f(x'))]}
\]
Therefore, we obtained that $N(f)=\mathcal{N}_1(f)=0$. On the other hand, $N(f^2)=4$ while $\mathcal{N}_2(f)=2$.

\end{example}

\bibliographystyle{plain}
%\bibliography{references}

\end{document}